\documentclass[12pt]{amsart}

\usepackage{verbatim, amssymb, enumerate}
\usepackage{hyperref}

\renewcommand \a{\alpha}
\renewcommand \b{\beta}

\newcommand \la{\lambda}

\newcommand \br{\mathbb{R}}
\newcommand \bc{\mathbb{C}}

\newcommand \rk{\operatorname{rk}}

\newcommand \Der{\operatorname{Der}}
\newcommand \End{\operatorname{End}}
\newcommand \Ric{\operatorname{Ric}}

\renewcommand \Re{\operatorname{Re}}

\newcommand \Span{\operatorname{Span}}
\newcommand \Tr{\operatorname{Tr}}

\newcommand \ig{\mathfrak{i}}
\newcommand\ag{\mathfrak a}
\newcommand\g{\mathfrak g}

\newcommand\z{\mathfrak z}

\newcommand \n{\mathfrak{n}}

\newcommand \ad{\operatorname{ad}}

\newcommand \diag{\operatorname{diag}}
\newcommand \Conv{\operatorname{Conv}}
\newcommand \grad{\operatorname{grad}}

\DeclareMathOperator{\ric}{Ric}

\newcommand \<{\langle}
\renewcommand \>{\rangle}
\newcommand \ip{\<\cdot,\cdot\>}

\makeatletter
\newtheorem*{rep@theorem}{\rep@title}
\newcommand{\newreptheorem}[2]{%
\newenvironment{rep#1}[1]{%
\def\rep@title{#2 \ref{##1}}%
\begin{rep@theorem}}%
{\end{rep@theorem}}}
\makeatother

\theoremstyle{plane}
\newtheorem{theorem}{Theorem}
\newtheorem*{theorem*}{Theorem}
\newtheorem*{corollary*}{Corollary}
\newtheorem*{conj*}{Conjecture}
\newtheorem{lemma}{Lemma}
\newtheorem{proposition}{Proposition}
\newtheorem*{prop*}{Proposition}
\newreptheorem{theorem}{Theorem}

\theoremstyle{definition}

\newtheorem*{definition*}{Definition}

\theoremstyle{remark}
\newtheorem{remark}{Remark}

\begin{document}

\title{Solvable extensions of negative Ricci curvature of filiform Lie groups} 

\author{Y.~Nikolayevsky}

\address{Department of Mathematics and Statistics, La Trobe University, Melbourne, Australia 3086}
\email{Y.Nikolayevsky@latrobe.edu.au}



\thanks{The author was partially supported by ARC Discovery Grant DP130103485. }

\begin{abstract}
We give necessary and sufficient conditions of the existence of a left-invariant metric of strictly negative Ricci curvature on a solvable Lie group the nilradical of whose Lie algebra $\g$ is a filiform Lie algebra $\n$. It turns out that such a metric always exists, except for in the two cases, when $\n$ is one of the algebras of rank two, $L_n$ or $Q_n$, and $\g$ is a one-dimensional extension of $\n$, in which cases the conditions are given in terms of certain linear inequalities for the eigenvalues of the extension derivation.
\end{abstract}

\maketitle

\section{Introduction}
\label{s:intro}

The question of whether (and when) a given manifold admits a Riemannian metric with a particular sign of the curvature is one of the fundamental in Riemannian geometry. Similarly, for homogeneous manifolds, the same question can be asked for left-invariant metrics. In that case the curvature is entirely expressed in terms of the algebraic structure of the given homogeneous space and one expects the answer to be stated in both topological and algebraic terms.

In this paper we continue the study of metric solvable Lie groups admitting a left-invariant metric of negative Ricci curvature, which has been started in \cite{NN}, and we refer the reader to the Introduction of that paper for a detailed overview of known results. At present, necessary and sufficient conditions for a homogeneous space to admit a left-invariant metric with a particular sign of the \emph{sectional} curvature are well understood, as well as the conditions for a homogeneous space to admit a left-invariant metric with positive or with zero \emph{Ricci} curvature. By the result of Milnor \cite{Mil} (for Lie groups) and Berestovskii \cite{Ber} (in the general case), a homogeneous space admits a left-invariant metric with $\Ric > 0$ if and only if it is compact and has a finite fundamental group (compare to the Myers Theorem). By the result of \cite{AK}, any Ricci-flat homogeneous space is flat. Note that flat homogeneous spaces were completely described in~\cite{Ale,BB}: every such space is isometric to a solvmanifold, the nilradical $\n$ of whose Lie algebra $\g$ is abelian, its orthogonal complement $\ag=\n^\perp$ is also abelian and the operators $\ad_Y, \; Y \in \ag$, are skew-symmetric.

Much less is known, however, about Riemannian homogeneous spaces of negative Ricci curvature. No unimodular solvable Lie group (in particular, no nilpotent group) admits a left-invariant metric with $\Ric < 0$, by the following result.
\begin{theorem}[\cite{DLM}]\label{t:dm}
A unimodular Lie group which admits a left-invariant metric with $\Ric < 0$ is noncompact and semisimple.
\end{theorem}
Examples of left-invariant metrics with $\Ric < 0$ were constructed on $\mathrm{SL}(n,\mathbb{R}), \; n \ge 3$, in \cite{LDM} and on some complex simple Lie groups in \cite{DLM}.

The general, non-unimodular, case however, seems to be wide open, even for Lie groups (leave alone homogeneous spaces). It is well known that the Ricci curvature of a left-invariant metric on a Lie group $G$ can be entirely computed from the algebraic data: the structure of the Lie algebra $\g$ of $G$ and the inner product $\ip$ on $\g$ (see Section~\ref{ss:ricci} for details). From this point on we descend to the level of Lie algebras and, with a slight abuse of terminology, will speak of the Ricci curvature of the metric Lie algebra $(\g, \ip)$.

In all the known cases, the necessary and sufficient condition for a given solvable Lie algebra $\g$ to admit an inner product with $\Ric < 0$ has the following form: $\g$ is an extension of its nilradical $\n$ by some derivations, one of which satisfies certain linear inequalities imposed on the real parts of its eigenvalues. The precise nature of such inequalities depends on the structure of $\n$ and in the general case remains unknown. In \cite{NN} the authors speculated that they may be related to the fact that (the real semisimple part of) the derivation belongs to a certain convex cone in the torus of derivations of $\n$.

For example, a solvable Lie algebra $\g$ with an abelian nilradical $\n$ admits an inner product of negative Ricci curvature if and only if there exists $Y \in \g$ such that all the eigenvalues of the restriction of $\ad_Y$ to $\n$ have positive real part. This is a consequence of the following Theorem.

\begin{theorem}[\cite{NN}]\label{t:neg}
Suppose $\g$ is a solvable Lie algebra. Let $\n$ be the nilradical of $\g$ and $\z$ be the centre of $\n$. Then
\begin{enumerate}[{\rm (1)}]
  \item \label{it:neg1}
  If $\g$ admits an inner product of negative Ricci curvature, then there exists $Y \in \g$ such that $\Tr \ad_Y > 0$ and all the eigenvalues of the restriction of the operator $\ad_Y$ to $\z$ have positive real part;

  \item \label{it:neg2}
  If there exists $Y \in \g$ such that all the eigenvalues of the restriction of $\ad_Y$ to $\n$ have positive real part, then $\g$ admits an inner product of negative Ricci curvature.
\end{enumerate}
\end{theorem}

The necessary and sufficient conditions to admit an inner product with $\Ric < 0$ for Lie algebras whose nilradical is a Heisenberg Lie algebra or is a standard filiform algebra have a similar ``flavour" \cite{NN}.

In this paper we study solvable algebras whose nilradical belongs to another important class of the variety of nilpotent Lie algebras -- the class of filiform algebras. \emph{Filiform} Lie algebras are those nilpotent algebras which are ``the least" nilpotent --  they have the maximal possible number of nonzero terms in the lower central series for a given dimension, that is, $\n^{(n-2)} \ne 0$, where $n=\dim \n$ \cite{Ver}. From among these algebras, there are two distinguished ones, $L_n$ and $Q_n$, which have rank two; all the other filiform algebras have rank one or zero by \cite{GK} -- see Section~\ref{ss:fili} for details.

The algebra $L_n=\Span(X_1, \dots, X_n)$ is defined by the relations $[X_1, X_i]=X_{i+1}, \; i=2, \dots, n-1$, where from now on we adopt the convention that all the relations between the basis elements of a Lie algebra which we do not list are zero (unless they follow from the given ones by the skew-symmetry). The codimension one abelian ideal $\ig=\Span(X_2, \dots, X_n)$ and the one-dimensional centre $\br X_n$ of $L_n$ are both characteristic ideals of $L_n$ (they are invariant under the action of any derivation on $L_n$). If $\g$ is a solvable extension of $L_n$, define the one-forms $\iota_2$ and $\iota_n$ on $\g$ as follows: for $Y \in \g$, $[Y,X_n]=\iota_n(Y) X_n$ and $\iota_2(Y)=\Tr((\ad_Y)_{|\ig})$.

The algebra $Q_n=\Span(X_1, \dots, X_n), \; n=2m$, is defined by the relations $[X_1, X_i]=X_{i+1}$, $i=2, \dots, n-2$, and $[X_j,X_{n-j+1}]=(-1)^{j+1}X_n, \; j=2, \dots, n-1$. The members $\ig_k=Q_n^{(k-2)}=\Span(X_{k}, \dots, X_n), \; k=3, \dots, n$, of the lower central series are all characteristic ideals of $Q_n$, so for a solvable extension $\g$ is of $Q_n$, we can define the one-forms $\iota_k, \; k=3, \dots, n$, on $\g$ as follows: for $Y \in \g, \; \iota_k(Y) = \Tr((\ad_Y)_{|\ig_k})$.

Our main result is the following theorem.

\begin{theorem}\label{th:fili}
Let $\g$ be a solvable non-nilpotent Lie algebra with the filiform nilradical $\n$. Then $\g$ admits an inner product of negative Ricci curvature if and only if
\begin{enumerate}[{\rm (a)}]
  \item\label{it:t1}
  either $\n=L_n$ and $\g$ is a one-dimensional extension of $\n$ by a vector $Y$ such that $\iota_2(Y), \iota_n(Y) > 0$ \cite[Theorem~4]{NN};

  \item\label{it:t2}
  or $\n=Q_n, \; n=2m, \; m>2$, and $\g$ is a one-dimensional extension of $\n$ by a vector $Y$ such that $\iota_k(Y) > 0, \; k=m+1, \dots, n$;

  \item\label{it:t3}
  or, with no restrictions, in all the other cases: either $\n$ is any other filiform algebra, or otherwise $\n$ is $L_n$ or $Q_n$ and $\g$ is an extension of $\n$ of dimension greater than one.
\end{enumerate}
\end{theorem}

\begin{remark}\label{rem:abun}
Note that the set of inequalities in \eqref{it:t2} is abundant -- for all of them to hold it is necessary and sufficient that just two of them hold, namely, $\iota_{n}(Y) > 0$ (the eigenvalue of $\ad_Y$ on the centre of $Q_n$) and one other $\iota_l(Y) > 0$. Due to the fact that this $l$ does not have a nice description we postpone this question till the end of Section~\ref{s:proof} (Theorem~\ref{t:Qn}).

Note also that Case~\eqref{it:t3} is not as ``universal" as it may sound: the truth is that ``the majority" of filiform algebras are characteristically nilpotent, hence admitting no non-nilpotent extensions at all -- see Section~\ref{ss:fili}.
\end{remark}

I would like to thank Yurii Nikonorov for his suggestion to study this topic and many useful ideas, Grant Cairns and Nguyen Thang Tung Le the work with whom stimulated my interest in the question, and Jorge Lauret for useful discussions, including the ones concerning the moment map (see the end of Section~\ref{s:proof}).

\section{Preliminaries}
\label{s:pre}

\subsection{Ricci tensor}
\label{ss:ricci}
Let $G$ be a Lie group with a left-invariant metric $Q$
obtained by the left translations from an inner product $\ip$ on the Lie algebra $\g$ of $G$.
Let $B$ be the Killing form of $\g$, and let $H \in \g$ be the \emph{mean curvature vector} defined by
$\<H, X\> = \Tr \ad_X$.

The Ricci curvature $\mathrm{ric}$ of the metric Lie group $(G,Q)$ at the identity is given by
\begin{equation} \label{eq:riccA1}
    \mathrm{ric}(X)=-\<[H,X],X\>-\frac12 B(X,X)-\frac12 \sum\nolimits_i \|[X,E_i]\|^2
    +\frac14 \sum\nolimits_{i,j} \<[E_i,E_j],X\>^2,
\end{equation}
for $X \in \g$, where $\{E_i\}$ is an orthonormal basis for $(\g, \ip)$ (see e.g. \cite{Ale} or \cite{Bes}).

Equivalently, one can define the Ricci operator $\ric$ of the metric Lie algebra $(\g, \ip)$ (the symmetric operator associated to $\mathrm{ric}$) by
\begin{equation} \label{eq:riccAl}
\ric = -\frac{1}{2} \sum\limits_i \ad_{E_i}^t \ad_{E_i} + \frac{1}{4} \sum\limits_i \ad_{E_i}\ad_{E_i}^t-\frac{1}{2}B -(\ad_H)^{s},
\end{equation}
where $(\ad_H)^{s}=\frac12(\ad_H+\ad_H^t)$ is the symmetric part of $\ad_H$.

In the case when $\g$ is solvable and is a one-dimensional extension of its nilradical $\n$ (this is the only case in this paper for which we will need an explicit formula for the Ricci tensor), one can simplify \eqref{eq:riccAl} further. Choose an orthonormal basis $\{e_i\}$ for $\n$ and a unit vector $f \perp \n$. Denote $T=\Tr \ad_f$. Note that $\g$ is unimodular if and only if $T=0$. Otherwise changing the sign of $f$ if necessary we get $T = \|H\|>0$.

Relative to the basis $\{e_1,...,e_n, f\}$, the matrix of the Ricci operator of the solvable metric Lie algebra $(\g, \ip)$ has the form (see the proof of \cite[Theorem~3]{NiN})
\begin{equation} \label{eq:ricmat}
\Ric = \left( {{\begin{array}{*{20}c}
 R_1 \hfill & R_2 \hfill\\
 R_2^t \hfill & r_3 \hfill\\
 \end{array} }} \right),
\end{equation}
where
\begin{align} \label{eq:r1}
R_1 &= \Ric^\mathfrak{n} + \frac{1}{2} [A,A^t] - T A^s, \\
(R_2)_j &= \frac{1}{2} \sum\limits_{i=1}^{n} \<[f,e_i],[e_i, e_j]\>, \quad j=1, \dots, n, \label{eq:r2}\\
r_3 &= - \Tr((A^s)^2), \label{eq:r3}
\end{align}
and $\Ric^\n$ is the matrix of the Ricci operator of the metric nilpotent Lie algebra $(\n, \ip_{\n})$ relative to the basis $\{e_1, \dots,e_n\}$. As $H = 0$ and $B = 0$ we get from \eqref{eq:riccA1}
\begin{equation}\label{eq:riccinilexplicit}
\<\ric^{\n} X, Y \> = \frac14 \sum_{i,j} \<X, [e_i, e_j]\> \<Y, [e_i, e_j]\> -\frac12 \sum_{i,j} \<[X, e_i], e_j\> \<[Y, e_i], e_j]\>.
\end{equation}
for $X, Y \in \n$.

\subsection{Filiform algebras and their extensions}
\label{ss:fili}

Filiform algebras are ``the least nilpotent" among nilpotent Lie algebras. They have been introduced by Michele Vergne in \cite{Ver} and studied extensively since then. In this paper we are interested mostly not in filiform algebras as such, but in the solvable extensions of them. These have been studied in depth and classified depending on the dimension of the (non-nilpotent) extension in \cite{Sun} or \cite{GK}. It turns out that ``most" filiform algebras are \emph{characteristically nilpotent}, so that any derivation of them is nilpotent.

Clearly, extending a nilpotent algebra by a nilpotent derivation (or more generally, extending by a subspace of derivations, some nonzero elements of which are nilpotent) increases the nilradical, and we do not allow this. Temporarily passing to the complexification, by \cite[Theorem~2]{Sun}, we get that a solvable algebra $\g$ having a filiform algebra $\n$ as its nilradical is an extension by commuting derivations. Furthermore, we define the \emph{rank} of a Lie algebra $\g$ to be the dimension of a maximal abelian subalgebra in the algebra of derivations $\Der(\g)$, all whose nonzero elements are semisimple (such a subalgebra is called a \emph{maximal torus} of derivations; all the maximal tori are conjugate by an automorphism and have the same dimension 
\cite{Che}). By \cite[Th\'{e}or\`{e}me~2]{GK} we have: 

\begin{lemma}\label{l:GK}
Suppose $\n$ is a filiform algebra over $\bc$ of positive rank. Then
\begin{enumerate}[{\rm 1.}]
  \item \label{it:GKrk2}
  Either $\rk \n =2$, in which case
  \begin{enumerate}[{\rm(a)}]
    \item
    either $\n$ is isomorphic to the algebra $L_n$ defined by the relations $[X_1,X_i]=X_{i+1}$, $2 \le i \le n-1$, with a maximal torus of derivations $\Span(\phi_1, \phi_2)$, where $\phi_1(X_i)=iX_i, \; 1 \le i \le n$, and $\phi_2(X_1)=0, \phi_2(X_i)=X_i, \; 2 \le i \le n$;
    \item
    or $\n$ is isomorphic to the algebra $Q_n, \; n=2m$, defined by the relations
    \begin{equation}\label{eq:relQn}
    [X_1,X_i]=X_{i+1}, \; 2 \le i \le n-2, \qquad [X_j, X_{n-j+1}] = (-1)^{j+1} X_n, \; 2 \le j \le n - 1,
    \end{equation}
    with a maximal torus of derivations $\Span(\phi_1, \phi_2)$, where
    \begin{equation}\label{eq:Qnder}
    \begin{gathered}
    \phi_1(X_i)=iX_i, \; 1 \le i \le n-1, \quad \phi_1(X_n)=(n+1)X_n\\  \phi_2(X_1)=0,  \quad \phi_2(X_i)=X_i, \; 2 \le i \le n-1, \quad  \phi_2(X_n)=2X_n.
    \end{gathered}
    \end{equation}
  \end{enumerate}

  \item
  Or $\rk \n =1$, in which case $\n$ is isomorphic to an algebra from the family $A_n^r(\alpha_1, \dots, \alpha_t)$, $1 \le r \le n-4$, defined by the relations of $L_n$ and some other relations, or from the family $B_n^r(\alpha_1, \dots, \alpha_t), \; n=2m, \; 1 \le r \le n-5$, defined by the relations of $Q_n$ and some other relations. In the both cases, $\alpha_1, \dots, \alpha_t$ are parameters satisfying certain algebraic equations, and a maximal torus of derivations is $\Span(\phi)$, where
  \begin{equation}\label{eq:ABder}
  \phi(X_1)=X_1,  \phi(X_i)=(i+r)X_i, \; 2 \le i \le n-1, \quad  \phi(X_n)=(n+2r)X_n.
  \end{equation}
\end{enumerate}
\end{lemma}

As usual, the relations between the basis elements for $L_n$ and for $Q_n$ which we do not list are zero (unless they follow from the given ones by the skew-symmetry). Note that in \cite{GK} there is one other class of algebras of rank one, $C_n$, but as it is shown in \cite[Remarque~1]{GV} all the algebras of class $C_n$ are isomorphic to $Q_n$.

We intentionally do not give all the details on the families $A_n^r$ and $B_n^r$ --- for our purposes we only need to know $\phi$.

\begin{remark}\label{rem:Qn}
Note that when considering the algebra $Q_n, \; n=2m$, both over $\br$ and over $\bc$, we can assume that $n \ge 4$, as $Q_4$ is given by the relations $[X_1,X_2]=X_3, \; [X_2, X_3]=-X_4$ and thus is isomorphic to $L_4$ with the basis $\{-X_2,X_1,X_3,X_4\}$. We also note that changing the basis $X_i$ for $Q_n$ to the basis $Y_1=X_1+X_2, \; Y_i=X_i, \, i > 1$, we get the relations $[Y_1,Y_i]=Y_{i+1}, \; 2 \le i \le n-1, \; [Y_i,Y_{n-i+1}]=(-1)^{i+1}Y_n, \, 2 \le i \le n-1$ (which differ from \eqref{eq:relQn} just by $[Y_1,Y_{n-1}=Y_n]$) that might sometimes be more convenient to use.
\end{remark}

\smallskip

Getting back to the real solvable Lie algebra $\g$ with the filiform nilradical $\n$ we need to exercise a certain caution, as two real nilpotent algebras can be isomorphic over $\bc$, without being isomorphic over $\br$ (this phenomenon can be observed even in the low dimension classification lists). The complexification $\n^\bc$ of $\n$ must be one of the algebras from Lemma~\ref{l:GK}. Now if $\rk \n^\bc = 1$, then $\g^\bc$ is the one-dimensional extension of $\n^\bc$ by a complex nonzero multiple of the derivation $\phi$ given by \eqref{eq:ABder}. It follows that $\g$ is the one-dimensional extension of $\n$ by a complex nonzero multiple of the derivation $\phi$. But as $\phi$ is a real linear operator and as all the eigenvalues of $\phi$ have different absolute value (so that no eigenvalues of $c \phi, c \in \bc \setminus \{0\}$, can possibly be complex conjugate), we obtain that $\g$ is the extension of $\n$ by a derivation all of whose eigenvalues are positive (up to changing the sign of $\phi$), that is, $\g$ admits an inner product of negative Ricci curvature by  Theorem~\ref{t:neg}\eqref{it:neg2}.

We now proceed to the algebras of rank two. In the next lemma, $L_n$ and $Q_n$ are complex Lie algebras given by the relations of Lemma~\ref{l:GK}\eqref{it:GKrk2}. We denote $L_n^\br$ and $Q_n^\br$ respectively the real Lie algebras defined by the same relations.
\begin{lemma}\label{l:real}
\begin{enumerate}[{\rm 1.}]
  \item \label{it:real1}
  Let $\n$ be a real, $n$-dimensional, filiform Lie algebra given by the relations $[Y_1,Y_i]=Y_{i+1}$, for $2 \le i \le n-1$ and $[Y_i,Y_j]=K_{ij}Y_n$ for $2 \le i, j \le n-1$, where $K=(K_{ij})$ is a nonsingular, skew-symmetric matrix. Then $\n$ is isomorphic to $Q_n^\br$.
  \item \label{it:real2}
  Let $\n$ be a real filiform Lie algebra whose complexification is isomorphic to $L_n$ (resp. $Q_n$) over $\bc$. Then $\n$ is isomorphic over $\br$ to the real algebra $L_n^\br$ (resp. $Q_n^\br$). 
\end{enumerate}
\end{lemma}
\begin{proof}
1. For $K$ to possibly be nonsingular $n$ must be even (note that the elements of $K$ are labelled by $i,j=2, \dots, n-1$). From the Jacobi identities it follows that $K_{i,j+1}+K_{i+1,j}=0$ for $2 \le i, j \le n-2$ and $K_{i+1,n-1}=0$ for $2 \le i \le n-2$. From this and the skew-symmetry we obtain $K_{ij}=0$ when $i+j > n+1$ or when $i+j$ is even. When $k=i+j$ is odd and $k \le n+1$ we get $K_{2,k-2}= -K_{3,k-3}= K_{4,k-4}=\dots = -K_{k-2,2}$ and in particular, $K_{2,n-1} \ne 0$ as $K$ is nonsingular. Introduce a new basis $Y_i'$ for $\n$ by putting $Y'_1=Y_1, \; Y_2'=Y_2+a_3Y_3+a_4Y_4+ \dots, \; Y'_{i+1}=[Y'_1,Y'_i], \; 2 \le i \le n-1$. Then $Y'_n=Y_n$ and $[Y'_i, Y'_j]=K'_{ij}Y'_n$ for a skew-symmetric matrix $K'$. The arguments similar to the ones above show that $K'_{ij}=0$ when $i+j > n+1$ or when $i+j$ is even and that $K_{2,k-2}= -K_{3,k-3}= K_{4,k-4}=\dots = -K_{k-2,2}$ for odd $k \le n+1$. Furthermore, $K'_{2,n-1}=K_{2,n-1}$ and then $K'_{2,n-3}=2a_4K_{2,n-1}+P_3(a_3), \; K'_{2,n-5}=2a_6K_{2,n-1}+P_5(a_3,a_4,a_5), \dots $, where $P_3, P_5, \dots$ are certain polynomials. As $K_{2,n-1} \ne 0$, we can successively choose $a_4, a_6, \dots$ in such a way that $K'_{2i}=0$ for all $i < n-1$. Then, relative to the basis $Y_i'$, the algebra $\n$ is given by the relations $[Y'_1, Y'_i]=Y'_{i+1} , \; 2 \le i \le n-1$, and $[Y_i',Y'_{n+1-i}]=(-1)^{i+1}aY'_n$ for $2 \le i \le n-1$ (and all the other brackets $[Y'_i, Y'_j]$ with $2 \le i < j \le n-1$ vanish), where $a = -K_{2,n-1} \ne 0$. Changing $Y'_i, \ i \ge 2$, to $a^{-1}Y'_i$ we get the relations for $Q_n^\br$ (as given in Remark~\ref{rem:Qn}).

2. In the both cases, $\n$ can be viewed as a real subspace of $\n^\bc$ which is closed under the Lie bracket.

In the case $\n^\bc = L_n$ choose two elements $Y_1=\sum_i a_iX_i, \; Y_2=\sum_i b_i X_i \in \n$ such that $a_1 \ne 0$ and $a_1b_2 - a_2b_1 \ne 0$ (this is always possible as the complexification of $\n$ must be the whole $L_n$). Replacing $Y_2$ by $Y_2- \Re(b_1/a_1)Y_1$ if necessary we can assume that $b_1/a_1$ is purely imaginary. Now the vectors $Y_3=[Y_1,Y_2]=(a_1b_2 - a_2b_1)X_3 + (\dots), \; Y_4=[Y_1,Y_3]=a_1(a_1b_2 - a_2b_1)X_4 + (\dots), \dots, Y_n=[Y_1,Y_{n-1}]= a_1^{n-3}(a_1b_2 - a_2b_1)X_n$ (where $(\dots)$ is a linear combination of the $X_i$ with higher subscripts) is a real basis for $\n$. Moreover, by construction, $[Y_1,Y_i]=Y_{i+1}$ for $2 \le i \le n-1$ and $[Y_i,Y_j]=0$ for $i,j \ge 3$. Furthermore, if $b_1 \ne 0$, then the vector $[Y_2,Y_3]=b_1(a_1b_2 - a_2b_1)X_4 + (\dots)=(b_1/a_1) Y_4  + (\dots)$ does not belong to $\n$, as $b_1/a_1$ is a nonzero imaginary number. It follows that $b_1=0$, hence $[Y_2,Y_i]=0$ for $i \ge 3$, as required.

In the case $\n^\bc = Q_n$, we start as above by constructing the basis $Y_i$ for $\n$ with $[Y_1,Y_i]=Y_{i+1}, \; 2 \le i \le n-1$, such that $Y_i=c_iX_i + (\dots)$, where $(\dots)$ is a linear combination of the $X_j, \; j>i$, and $c_i \ne 0$. In particular, $Y_n=c_n X_n, \; c_n \ne 0$, lies in the centre of $\n$. Then $[Y_i,Y_j]=K_{ij}Y_n$ for $2 \le i, j \le n-1$. The matrix $K=(K_{ij})$ is real and skew-symmetric. It cannot be singular, as otherwise the real span of the vectors $Y_2, \dots, Y_n$ in $\n$ would be a subalgebra with the centre of dimension greater than one, hence the same would be true for their complex span which is the subalgebra $\Span(X_2, \dots, X_n) \subset \n^\bc$. But the centre of this subalgebra has dimension one. It follows that $K$ is nonsingular and so the claim follows from assertion~\ref{it:real1}.
\end{proof}

From Lemma~\ref{l:GK}\eqref{it:GKrk2} and Lemma~\ref{it:real2}\eqref{it:real2} it follows that we only need to prove Theorem~\ref{th:fili} for solvable extensions of the filiform algebras $L_n^\br$ and $Q_n^\br$. The proof for $L_n^\br$ is given in \cite[Theorem~4]{NN}, so it remains only to consider the case $\n=Q_n^\br, \; n=2m$, where we can assume that $m> 2$ by Remark~\ref{rem:Qn} and where, from now on, we will drop the superscript $\br$, so that $Q_n$ denotes the real Lie algebra given by the relations~\eqref{eq:relQn}.

We will need a slightly more detailed (as compared to \eqref{eq:Qnder}) knowledge of the derivations of $Q_n$ given in the following Lemma.

\begin{lemma}\label{l:derQn}
Relative to the basis $X_i$ from \eqref{eq:relQn}, the matrix of any derivation of $Q_n$ is given by
\begin{equation}\label{eq:derQn}
    M=\left(
\begin{array}{c@{}c@{}}
 \begin{array}{ccc}
         a & b& \\
         0& d & \\
        & & a+d \\
  \end{array} & \mbox{\Huge $0$ } \\
  \mbox{\Huge $\ast$ } & \begin{array}{cccc}
                       2a+d & & &\\ & \ddots & &\\ & & (n-3)a+d & \\ & & & (n-3)a+2d \\
                      \end{array}
\end{array}\right),
\end{equation}
for some $a, d, b \in \br$ and possibly some nonzero elements below the diagonal marked by the asterisk.
\end{lemma}
\begin{proof}
As the members of the lower central series are invariant with respect to a derivation we obtain that all the entries of $M$ above the diagonal in the columns starting from the third are zeros. Taking $M_{11}=a, \; M_{22}=d$ we can find all the diagonal entries from the relations \eqref{eq:relQn}. The fact that $M_{21}=0$ (which we will not need) follows by applying the derivation to $[X_1,X_{n-1}]=0$.
\end{proof}

\section{Proof of Theorem~\ref{th:fili}}
\label{s:proof}

As we have seen in Section~\ref{ss:fili}, we only need to consider the case when $\g$ is a one-dimensional extension of $\n=Q_n, \; n= 2m, \; m > 2$, by a non-nilpotent derivation. Otherwise, the case when $\g$ is a one-dimensional extension of $L_n$ was treated in \cite[Theorem~4]{NN} (this proves Case~\eqref{it:t1} of Theorem~\ref{th:fili}), and in all the other cases when $\g$ is a non-nilpotent solvable Lie algebra having a filiform nilradical $\n$, it follows from Lemma~\ref{l:GK} and the arguments in Section~\ref{ss:fili} that for some $Y \in \g$ by which we are extending $\n$, the restriction of the derivation $\ad_Y$ to $\n$ has all eigenvalues positive, and so an inner product with $\Ric < 0$ exists by Theorem~\ref{t:neg}\eqref{it:neg2}.

From Lemma~\ref{l:derQn} we see that the eigenvalues of any derivation of $Q_n$ are $a,d,a+d, 2a+d, \dots, (n-3)a+d, (n-3)+2d$. Now if we extend $\n=Q_n$ by at least three derivations, then one of them will be nilpotent, which we do not allow. If we extend $Q_n$ by exactly two derivations no nonzero linear combination of which is nilpotent, then there will exist a linear combination of them with all the eigenvalues positive (say the one with $a=d=1$). Then an inner product with $\Ric < 0$ exists by Theorem~\ref{t:neg}\eqref{it:neg2}. This, combined with the argument in the previous paragraph, completes the proof of Case~\eqref{it:t3} of Theorem~\ref{th:fili}.

We can therefore assume that $\g$ is an extension of $\n$ by exactly one derivation (of the form \eqref{eq:derQn}), which is not nilpotent, that is, either $a$ or $d$ is nonzero. What is more, as a unimodular $\g$ does not admit an inner product with $\Ric < 0$ by Theorem~\ref{t:dm}, we can assume (changing the sign if necessary) that $T=\Tr M= \frac12 (n-1)(n-2)a + nd > 0$. Computing the numbers $\iota_k$ from Case~\eqref{it:t2} of Theorem~\ref{th:fili} for such a derivation we obtain that the only claim which remains to be proved is equivalent to the following statement. 

\begin{proposition}\label{p:Qn} 
Let $\g$ be a one-dimensional extension of $\n=Q_n, \; n = 2m, m > 2$, by a derivation \eqref{eq:derQn}. There exists an inner product on $\g$ with $\Ric < 0$ if and only if
\begin{equation}\label{eq:iotak}
    \iota_k=\Big((n-3)+\sum_{i=k-2}^{n-3}i\Big)a + (n-k+2)d = \frac12((n-3)n-(k-3)(k-2))a + (n-k+2)d > 0,
\end{equation}
for all $k = m+1, \dots, n$ (up to changing the sign of $M$ if necessary).
\end{proposition}
\begin{remark}\label{rem:sign}
Note that inequalities~\eqref{eq:iotak} in fact imply $T > 0$, as $T=\frac{2}{n-3}\iota_{n-1}+\frac{n^2-3n-6}{2(n-3)}\iota_n$, with both coefficients being positive when $n > 4$.
\end{remark}
\begin{proof}
\textbf{Necessity.} Let $\g$ be the one-dimensional extension of $\n=Q_n$ by a derivation $\phi$ whose matrix relative to the basis $X_i$ satisfying \eqref{eq:relQn} is given by \eqref{eq:derQn}. Suppose that $\ip$ is an inner product on $\g$ for which $\Ric < 0$. Let $f$ be a unit vector orthogonal to $\n$. Note that relative to the (in general, nonorthonormal) basis $X_i$ for $\n$ the matrix of the restriction of $\ad_f$ to $\n$ is proportional to $M$ (some of the entries below the diagonal may change, but we don't care). Up to scaling the inner product and changing the sign of $f$ we can assume that $(\ad_f)_{\n}=M$ with $T=\Tr M= \frac12 (n-1)(n-2)a + nd > 0$. If $a=d$, then the inequalities \eqref{eq:iotak} are satisfied and there is nothing to prove. Assuming $a \ne d$ we can modify the basis $X_i$ by changing $X_2$ to $X_2-b(a-d)^{-1}X_1$ to eliminate $b$ from $M$, still keeping the relations \eqref{eq:relQn} unchanged. We keep the notation $X_i$ for this new basis. Finally, let $e_n, e_{n-1}, \dots, e_2, e_1$ be the orthonormal basis for $\n$ constructed by the Gram-Schmidt procedure from the basis $X_n, X_{n-1}, \dots, X_2, X_1$. We have the following Lemma.

{
\begin{lemma}\label{l:derQnorth}
{\ }

\begin{enumerate}[{\rm 1.}]
  \item \label{it:dQ1}
  For $i \ge 1, \quad \Span(X_i, \dots, X_n)=\Span(e_i, \dots, e_n)$, in particular $e_n$ is a nonzero multiple of $X_n$ and spans the centre of $\n$ and $\Span(e_2, \dots, e_n)=\Span(X_2, \dots, X_n)$.

  \item \label{it:dQ2}
  Relative to the orthonormal basis $e_i$ for $\n$, the matrix $A$ of the restriction of $\ad_f$ to $\n$ is lower-triangular, with the same diagonal entries as $M$, that is, $\la_1=a, \; \la_i=d+(i-2)a$ for $2 \le i \le n-1$, and $\la_n=2d+(n-3)a$ (from the top left to the bottom right corner).

  \item \label{it:dQ3}
  For $2 \le i \le n-2, \quad [e_1,e_i]=c_ie_{i+1}+e'_{i+1}$, where $c_i \ne 0$ and $e'_{i+1} \in \Span(e_{i+2}, \dots, e_n)$. For $2 \le i,j \le n-1, \quad [e_i,e_j]=\<Ke_i, e_j\>e_n$, where $K$ is a skew-symmetric matrix, with $[e_i,e_j]=0$ for $2 \le i,j \le n-1, \; i+j > n+1$.
\end{enumerate}
\end{lemma}
\begin{proof}
Assertion~\ref{it:dQ1} follows directly from our construction. Assertion~\ref{it:dQ2} follows from assertion~\ref{it:dQ1} and from the fact that $M$ is lower-triangular. The first two statements of assertion~\ref{it:dQ1} also follow from assertion~\ref{it:dQ1} and relations \eqref{eq:relQn}. The fact that $[e_i,e_j]=0$ for $2 \le i,j \le n-1, \; i+j > n+1$, can be proved similar to that in the proof of Lemma~\ref{l:real}\eqref{it:real1}: from the Jacobi identities it follows that $K_{i,j+1}+K_{i+1,j}=0$ for $2 \le i, j \le n-2$ and $K_{i+1,n-1}=0$ for $2 \le i \le n-2$, which implies the claim.
\end{proof}
}

By Lemma~\ref{l:derQnorth}\eqref{it:dQ1} the subspaces $\ig_k$ defined before Theorem~\ref{th:fili} coincide with the subspaces $\Span(e_k, \dots, e_n)$. Denote $\pi_k \in \End(\n)$ the orthogonal projection to $\ig_k$. The proof of the necessity is based on estimating $\Tr \pi_k R_1$.
From Lemma~\ref{l:derQnorth}\eqref{it:dQ2} we have
\begin{equation}\label{eq:trpiVAAt}
    \Tr (\pi_k [A,A^t])=\sum_{j=k}^n (\|A^te_j\|^2 - \|Ae_j\|^2)= \sum_{j=k}^n \sum_{i=1}^m A_{ji}^2 \ge 0,
\end{equation}
as $A$ is lower-triangular. Furthermore, again by Lemma~\ref{l:derQnorth}\eqref{it:dQ2},
\begin{equation}\label{eq:trpiVAs}
    \Tr (\pi_k A^s) = \Tr (\pi_k A) = \sum_{j=k}^n \la_j= \iota_k.
\end{equation}
From \eqref{eq:riccinilexplicit} we obtain
\begin{equation}\label{eq:ricnenen}
\begin{split}
\<\ric^{\n} e_n,e_n\> &= \frac{1}{4} \sum_{i,j} \<[e_i,e_j],e_n\>^2\\
&= \frac{1}{2} \sum_{j=2}^{n-1} \<[e_1,e_j],e_n\>^2 + \frac{1}{4} \sum_{i,j=2}^{k-1} \<[e_i,e_j],e_n\>^2 + \frac{1}{2} \sum_{i=k}^{n-1} \sum_{j=2}^{n-1} \<[e_i,e_j],e_n\>^2,
\end{split}
\end{equation}
where we used the fact that $k  \ge m+1$ and that $[e_i,e_j]=0$ for $2 \le i,j \le n-1$ with $i+j > n+1$ from Lemma~\ref{l:derQnorth}\eqref{it:dQ3}. Again from \eqref{eq:riccinilexplicit}, with $l=k, \dots, n-1$, we have
\begin{equation*}
\begin{split}
\<\ric^{\n} e_l,e_l\> &= -\frac{1}{2} \sum_i \|[e_i,e_l]\|^2 + \frac{1}{4} \sum_{i,j} \<[e_i,e_j],e_l\>^2 \\
&= -\frac{1}{2} \Big(\|[e_1,e_l]\|^2 + \sum_{i=2}^{n-1} \|[e_i,e_l]\|^2 \Big)+ \frac{1}{2} \sum_{j=2}^{n-1} \<[e_1,e_j],e_l\>^2\\
&= -\frac{1}{2} \sum_{l<j \le n} \<[e_1,e_l],e_j\>^2 -\frac{1}{2} \sum_{i=2}^{n-1} \<[e_i,e_l],e_n\>^2 + \frac{1}{2} \sum_{2 \le j <l} \<[e_1,e_j],e_l\>^2,
\end{split}
\end{equation*}
where in the last line we used Lemma~\ref{l:derQnorth}\eqref{it:dQ3}. From this and \eqref{eq:ricnenen} we now obtain
\begin{equation*}
\begin{split}
\Tr \pi_k \ric^{\n} &= \<\ric^{\n} e_n,e_n\>+ \sum_{l=k}^{n-1}\<\ric^{\n} e_l,e_l\>= \frac{1}{4} \sum_{i,j=2}^{k-1} \<[e_i,e_j],e_n\>^2 \\
&+ \frac{1}{2} \sum_{j=2}^{n-1} \<[e_1,e_j],e_n\>^2-\frac{1}{2} \sum_{l=k}^{n-1}\sum_{j=l+1}^{n} \<[e_1,e_l],e_j\>^2 + \frac{1}{2} \sum_{l=k}^{n-1} \sum_{j=2}^{l-1} \<[e_1,e_j],e_l\>^2 \\
&= \frac{1}{4} \sum_{i,j=2}^{k-1} \<[e_i,e_j],e_n\>^2 -\frac{1}{2} \sum_{j=k}^{n-1}\sum_{l=j+1}^n \<[e_1,e_j],e_l\>^2 + \frac{1}{2} \sum_{l=k}^{n} \sum_{j=2}^{l-1} \<[e_1,e_j],e_l\>^2\\
&= \frac{1}{4} \sum_{i,j=2}^{k-1} \<[e_i,e_j],e_n\>^2 + \frac{1}{2} \sum_{l=k}^{n} \sum_{j= 2}^{k-1} \<[e_1,e_j],e_k\>^2 \ge 0,
\end{split}
\end{equation*}
Using this together with (\ref{eq:trpiVAAt}, \ref{eq:trpiVAs}) we find from \eqref{eq:r1}
\begin{equation*}
    \Tr \pi_k R_1 = \Tr \pi_k(\ric^\n + \frac{1}{2} [A,A^t] - T A_1^s) \ge -T \iota_k.
\end{equation*}
As $\Ric < 0$, the left-hand side must be negative, which proves the necessity of~\eqref{eq:iotak}.

\medskip

\textbf{Sufficiency.} Suppose $\g=\br Y \oplus \n$ is a one-dimensional extension of $\n=Q_n$ by a derivation $M$ (which must be of the form \eqref{eq:derQn} relative to the basis $X_i$ by Lemma~\ref{l:derQn}) such that inequalities~\eqref{eq:iotak} are satisfied. We want to construct an inner product of negative Ricci curvature on $\g$.

First, if $a=d$, then all the eigenvalues of $M$ are positive (up to changing the sign of $M$) and the existence of a required inner product follows from Theorem~\ref{t:neg}\eqref{it:neg2}. Assuming $a \ne d$, as in the proof of the necessity above, we can modify the basis $X_i$ by changing $X_2$ to $X_2-b(a-d)^{-1}X_1$ to eliminate $b$ from $M$, still keeping the relations \eqref{eq:relQn} unchanged. We keep the notation $X_i$ for this new basis and the notation $M$ for the matrix of our derivation, which is now lower-triangular, with the same diagonal entries, namely $\la_1=a$, $\la_i=d+(i-2)a$ for $2 \le i \le n-1$, and $\la_n=2d+(n-3)a$ (from the top left to the bottom right corner). Next, let $N$ be a positive derivation of $\n$ which is diagonal relative to the basis $X_i$, say $N = \diag(1,2, \dots,n-1,n+1)$ (derivation $\phi_1$ in \eqref{eq:Qnder}). For $s > 0$ define the operator $T_s \in \End(\g)$ by $T_sY=Y$ and $T_sX=e^{sN}X$, for $X \in \n$. When $s \to \infty$, the Lie algebra $\g$ degenerates to the Lie algebra $\bar g$ with the same nilradical $\n=Q_n$ and with the restriction of $\ad_Y$ to $\n$ being diagonal, namely $(\ad_Y)_{|\n} = \diag(\la_1,\la_2, \dots, \la_n)$. By \cite[Proposition~1]{NN} it suffices to construct an inner product of negative Ricci curvature on $\bar \g$. In the notation of Section~\ref{ss:ricci} take $f=Y$ and $e_i=e^{-x_i}X_i, \; x_i \in \br$, for $i=1, \dots, n$, to be an orthonormal basis for the inner product which we are constructing.

We have $[f,e_i]=\la_i e_i$ for $i=1, \dots, n$, $[e_1,e_i]=e^{-x_1-x_i+x_{i+1}}e_{i+1}$ for $i=2, \dots, n-2$, and $[e_i,e_{n+1-i}]=(-1)^{i+1} e^{-x_i-x_{n+1-i}+x_n}e_{n}$ for $i=2, \dots, n-1$ (with all the other brackets being zero unless they follow from the ones above by the skew-symmetry). In the notation of Section~\ref{ss:ricci} we now have
\begin{equation*}
    A=A^s=\diag(\la_1,\la_2, \dots, \la_n), \quad R_2=0, \quad r_3=-\sum\nolimits_i \la_i^2 < 0,
\end{equation*}
so $\Ric < 0$ is equivalent to $R_1<0$, where $R_1=\Ric^\n - T A$ and $\Ric^\n$ by \eqref{eq:riccinilexplicit} is diagonal with the diagonal entries
\begin{gather*}
    (\Ric^\n)_{11}=-\frac12 \sum_{i=2}^{n-2}e^{2(-x_1-x_i+x_{i+1})}, \quad (\Ric^\n)_{22}=-\frac12 e^{2(-x_1-x_2+x_3)} -\frac12 e^{2(-x_2-x_{n-1}+x_n)}, \\
    (\Ric^\n)_{kk}=\frac12 e^{2(-x_1-x_{k-1}+x_k)} -\frac12 e^{2(-x_1-x_k+x_{k+1})}-\frac12 e^{2(-x_k-x_{n-k+1}+x_n)}, \quad 3 \le k \le n-2, \\
    (\Ric^\n)_{n-1,n-1} \! =\frac12 e^{2(-x_1-x_{n-2}+x_{n-1})} - \frac12 e^{2(-x_2-x_{n-1}+x_n)}, \, (\Ric^\n)_{nn} \! =\frac14 \sum_{i=2}^{n-1} e^{2(-x_{n+1-i}-x_i+x_n)}.
\end{gather*}
It follows that $R_1$ is diagonal with the diagonal entries constructed as follows. In a Euclidean space $\br^n$ with the inner product $(\cdot,\cdot)$ and with an orthonormal basis $E_i$, introduce the vectors $F_1=-E_1-E_2+E_3, \, F_2=-E_1-E_3+E_4, \dots, F_{n-3}=-E_1-E_{n-2}+E_{n-1}, \, F_{n-2}=-E_2-E_{n-1}+E_n, \, F_{n-1}=-E_3-E_{n-2}+E_n, \dots, F_{n-4+m}= -E_m-E_{m+1}+E_n$ (so that we have one vector $F_\a=-E_i-E_j+E_k$ per every triple of subscripts $(i,j,k), \; i<j$, such that $[e_i,e_j]$ is a nonzero multiple of $e_k$)
and the vectors $V_1=E_1+\sum_{i=3}^{n-1} (i-2)E_i+ (n-3)E_n, \; V_2=\sum_{i=2}^{n-1} E_i + 2E_n$. Then $R_1$ is the diagonal matrix whose diagonal entries are the corresponding components of the vector
\begin{gather} \label{eq:Px}
    P(x)=\frac12 \sum_{\a=1}^{n-4+m} e^{2(x,F_\a)}F_\a - T( a V_1 + d V_2) = \frac14 \grad (\phi(x)) - T( a V_1 + d V_2), \\
    \text{where } \phi(x)=\sum_{\a=1}^{n-4+m} e^{2(x,F_\a)}, \quad x \in \br^n, \label{eq:phix}
\end{gather}
and where $\grad \phi(x)$ is the vector dual to the one-form $d\phi$. Note that $T=(aV_1+ d V_2, \mathbf{1})$, where $\mathbf{1} \in \br^n$ is the vector all of whose components equal $1$.

We need the following technical Lemma. For a set $V$ of vectors $v_\a, \; 1 \le \a \le q$, in a Euclidean space $\br^N$ with the inner product $(\cdot, \cdot)$, introduce the function $f: \br^N \to \br$ by $f(x)=\sum_{\a=1}^q e^{(v_\a,x)}$. Denote $\Conv(V)$ the relative interior of the cone over the convex hull of the vectors $v_\a$, that is, the set of all linear combinations $\sum_\a \mu_\a v_\a$ with all the $\mu_\a$ being positive.

\begin{lemma}\label{l:conv}
In the above notation, $\{\grad f(x) \, : \, x \in \br^N\} = \Conv(V)$.
\end{lemma}
\begin{proof}
Without loss of generality we can assume that $\Span(V) =\br^N$ (otherwise we simply replace $\br^N$ with $\Span(V)$).

The inclusion $\{\grad f(x) \, : \, x \in \br^N\} \subset \Conv(V)$ is obvious. To establish the reverse inclusion choose a vector $u = \sum_\a \mu_\a v_\a \in \Conv(V), \; \mu_\a > 0$, and consider a function $g: \br^N \to \br$ defined by $g(x)=f(x)-(u,x)$. Then $\grad g = \grad f - u$, so to prove that $u \in \{\grad f(x) \, : \, x \in \br^N\}$ it suffices to show that $g$ has a critical point. We will show that $g$ in fact attains its minimum on $\br^N$. We have $g(0)=f(0)=q$, so it suffices to show that outside a compact in $\br^N$ the function $g$ is greater than or equal to $M:=q+1 > 0$. To construct such a compact introduce the quadratic form $Q(x)=\sum_\a (v_\a,x)^2$. As $V$ spans the whole space $\br^N$, $Q$ is positive definite. Our compact $C$ will be the set $\{ x \in \br^N \, : \, Q(x) \le R^2\}$, where $R$ is a large positive number which we will specify a little later.

Choose an arbitrary $x \notin C$. We want to show that $g(x) \ge M$, if $R$ is chosen to be large enough. First, if $(u,x) \le -M$, then $g(x) \ge M$. Suppose that $(u,x) \ge -M$, so that $\sum_\a \mu_\a (v_\a,x) \ge -M$. Let $S_+$ (respectively $S_-$) be the set of those $\a$ for which $(v_\a,x) \ge 0$ (resp. $(v_\a,x) < 0$) and let $\Sigma_\pm = \sum_{\a \in S_\pm} \mu_\a (v_\a,x)$. We have $\Sigma_+ \ge 0 \ge \Sigma_-$ and $\Sigma_+ +\Sigma_- \ge -M$. It follows that $0 \le -\Sigma_- \le \Sigma_+ + M$, so $0 \le \sum_{\a \in S_-} \mu_\a (-(v_\a,x)) \le \Sigma_+ + M$, from which $0 \le -(v_\a,x) \le \mu_\a^{-1}(\Sigma_+ + M)$ for all $\a \in S_-$, hence $\sum_{\a \in S_-} (v_\a,x)^2 \le (\Sigma_+ + M)^2 \sum_{\a \in S_-} \mu_a^{-2}$. Furthermore, from $\sum_{\a \in S_+} \mu_\a (v_\a,x) = \Sigma_+$ we get $\sum_{\a \in S_+} (v_\a,x)^2 \le \Sigma_+^2 \sum_{\a \in S_+} \mu_a^{-2}$. It follows that $R^2=Q(x) \le (\Sigma_+ + M)^2 \sum_{\a=1}^q \mu_a^{-2}$, so $\Sigma_+ \ge R (\sum_{\a=1}^q \mu_a^{-2})^{-1/2}-M$ (which is positive for $R$ large enough). But then, as $\Sigma_+ = \sum_{\a \in S_+} \mu_\a (v_\a,x)$, there exists a $\b \in S_+$ such that $(v_\b,x) \ge \Sigma_+ (\sum_{\a \in S_+} \mu_\a)^{-1} \ge (R (\sum_{\a=1}^q \mu_a^{-2})^{-1/2}-M)(\sum_{\a=1}^q \mu_\a)^{-1}$. On the other hand, by Cauchy-Schwartz, $(u,x) \le |\sum_\a \mu_\a (v_\a,x)| \le (\sum_{\a=1}^q \mu_a^2)^{1/2} R$, so we obtain $g(x) \ge \mu_\b e^{(v_\b,x)} - (u,x) \ge \mu_\b \exp((R (\sum_{\a=1}^q \mu_a^{-2})^{-1/2}-M)(\sum_{\a=1}^q \mu_\a)^{-1}) - (\sum_{\a=1}^q \mu_a^2)^{1/2} R$, which can be made greater than $M$ for all $\b =1, \dots, q$ if we choose $R$ to be large enough.

Then $g$ attains its minimum somewhere in $C$, so $u \in \Conv(V)$, as required.
\end{proof}

Returning to the proof, we need to find $x \in \br^n$ such that all the components of the vector $P(x)$ given by~\eqref{eq:Px} are negative, that is, such that $-P(x)$ belongs to the first octant of $\br^n$. This is equivalent to the existence of $x \in \br^n$ such that $T( a V_1 + d V_2) \in \frac14 \grad (\phi(x)) + \Conv(E_1, \dots, E_n)$, which by Lemma~\ref{l:conv} is equivalent to the fact that $a V_1 + d V_2 \in \Conv(E_1, \dots, E_n, F_1, \dots, F_{n-4+m})$ (as $T > 0$ by Remark~\ref{rem:sign}).

This reduces our proof to the following question in convex geometry (or rather even in linear programming): we need to show that if $a$ and $d$ satisfy \eqref{eq:iotak}, then $a V_1 + d V_2 \in \Conv(E_1, \dots, E_n, F_1, \dots, F_{n-4+m})$. Note that $E_n=F_{n-4+m}+E_m+E_{m+1}$, so crossing out $E_n$ does not change the convex cone on the right-hand side, and then successively, $E_{n-1}=F_{n-3}+E_1+E_{n-2}, \; E_{n-2}=F_{n-4}+E_1+E_{n-3}, \dots, E_3=F_1+E_1+E_2$, so we effectively need to show that $a V_1 + d V_2 \in \Conv(E_1, E_2, F_1, \dots, F_{n-4+m})$, that is, that $a V_1 + d V_2 =\sum_{i=1}^{n-2} w_i F_i + \sum_{j=1}^{m-2} y_j F_{n-2+j} + z_1E_1 + z_2E_2$, for some $w_i, y_j, z_1, z_2 > 0$.

We first of all note that $V_1, V_2 \perp F_\a$, for all $\a=1, \dots, n+m-4$, from which $z_1=a\|V_1\|^2+d (V_1,V_2), \; z_2=a(V_1,V_2)+d\|V_2\|^2$. Computing these we find $z_2=(m+1)\iota_n > 0$, $z_1=\frac {{n}^{3}-6\,{n}^{2}+11\,n+6}{6(n-3)}\iota_{n-1}+\frac {{n}^{2}-7\,n+6}{2(n-3)}\iota_n > 0$. Furtermore, as no nonzero linear combination of the vectors $F_\a$ lies in $\Span(E_1,E_2)$, it suffices to find $w_i, y_j > 0$ such that $a V_1' + d V_2' =\sum_{i=1}^{n-2} w_i F_i' + \sum_{j=1}^{m-2} y_j F_{n-2+j}'$, where dash means the orthogonal projection to $\Span(E_3, \dots, E_n)$. Acting on the both sides by the $(n-2) \times (n-2)$-matrix $T$ with the entries $T_{rs}=1$ if $s \ge r$, and $T_{rs}=0$ if $s<r$ we get the following system of linear equations
\begin{equation}\label{eq:hull}
\begin{gathered}
    \iota_3=w_1-\sum_{j=1}^{m-2}y_j, \quad \iota_4=w_2-\sum_{j=2}^{m-2}y_j, \quad \dots \quad , \iota_m=w_{m-2}-y_{m-2},\\
    \iota_{m+1}=w_{m-1},\\
    \iota_{m+2}=w_m+ y_{m-2}, \dots, \iota_{n-2}=w_{n-4}+\sum_{j=2}^{m-2}y_j, \quad \iota_{n-1}=w_{n-3}+\sum_{j=1}^{m-2}y_j\\
    \iota_{n}=w_{n-2}+\sum_{j=1}^{m-2}y_j.
\end{gathered}
\end{equation}
Assuming \emph{all} the numbers $\iota_3, \iota_4, \dots, \iota_n$ to be positive we can easily find positive $y_i, w_j$ satisfying \eqref{eq:hull} by first choosing $y_i > 0$ to be small enough and then finding $w_j$. It therefore remains to show that the inequalities $\iota_{m+1},\iota_{m+2}, \dots, \iota_n > 0$ imply that also $\iota_3, \iota_4, \dots, \iota_m > 0$.

From \eqref{eq:iotak}, every inequality $\iota_k > 0$ is equivalent to the inequality $\kappa_k a + d > 0$, where $\kappa_k=\frac{(n-3)n-(k-3)(k-2)}{2(n-k+2)}$, hence the system of inequalities $\iota_3, \iota_4, \dots, \iota_n > 0$ is equivalent to just two of them: $\kappa_{\min} a + d > 0, \; \kappa_{\max} a + d > 0$, where $\kappa_{\max} = \max \kappa_k, \; \kappa_{\min} = \min \kappa_k$. Now the function $f(t)=\frac{(n-3)n-(t-3)(t-2)}{2(n-t+2)}$, for $t \in [3,n]$, increases on $[3, n+2-\sqrt{2n}]$ and decreases on $[n+2-\sqrt{2n},n]$. It follows that $\kappa_{\max}= \kappa_l = \max(\kappa_{[n+2-\sqrt{2n}]}, \kappa_{[n+3-\sqrt{2n}]})$ and $\kappa_{\min}=\min(\kappa_3, \kappa_n)=\kappa_n=\frac12(n-3)$. Thus the system of inequalities $\iota_3, \iota_4, \dots, \iota_n > 0$ is equivalent to two inequalities: $\iota_l > 0, \; \iota_n > 0$ (the later one is simply $(n-3)a + 2d > 0$). As $l \ge m+1$, the claim follows.
\end{proof}

As we can see from the proof, the system of inequalities $\iota_{m+1}, \iota_{m+2}, \dots, \iota_n > 0$ is equivalent to just two of them (which should not be surprising, as the torus of derivations has dimension two). We can now sharpen the statement of Case~\eqref{it:t2} of Theorem~\ref{th:fili} as follows.

\begin{theorem}\label{t:Qn}
Let $\g$ be a solvable non-nilpotent Lie algebra which is a one-dimensional extension of its nilradical $\n=Q_n, \; n=2m, \; m>2$. Then $\g$ admits an inner product of negative Ricci curvature if and only if $\g$ is the extension of $\n$ by a vector $Y$ such that $\iota_n(Y) > 0, \; \iota_l(Y) > 0$, where the number $l$ is defined as follows: $l \in\{p,p+1\}$ and $f(l)=\max(f(p),f(p+1))$, where $p=[n+2-\sqrt{2n}]$ and $f(t)=\frac{(n-3)n-(t-3)(t-2)}{2(n-t+2)}$.

\end{theorem}

We finish this section and the paper with the observation which may be useful in the future study of solvable algebras with negative Ricci curvature. Analyzing our proof and the approach in \cite{NN} one can see that the core of the argument is the study of the behaviour of $\Ric^\n$. It looks very promising to use for that the \emph{moment map} $\mathbf{m}$ of the action of a certain group (say $\mathrm{SL}(n)$) on the variety of nilpotent Lie brackets. This already proved to be very effective in the study of Einstein nilradicals. For example, it is proved in \cite{Lau} that for a Lie bracket $\mu$, $\mathbf{m}([\mu]) = 4 \|\mu\|^{-2} \Ric$ (see \cite{Lau} or \cite{Nik} for unexplained terminology). Moreover, our Lemma~\ref{l:conv} is in fact a version of \cite[Lemma~2]{Nik}, which in turn follows from deep results on the convexity of the image of the moment map (see e.g. \cite{HS}).


\end{document}